\documentclass[a4paper,11pt,reqno]{customart}

\usepackage[utf8x]{inputenc}
\usepackage[margin=1in]{geometry}
\usepackage{verbatim}
\usepackage{color}
\usepackage[table]{xcolor}
\usepackage{listings}
\usepackage{amssymb,amsfonts,amsthm,amsmath}
\usepackage{url}
\usepackage{enumerate}
\usepackage{todonotes}
\usepackage[all,cmtip]{xy}
\usepackage{multirow}

\newtheorem{theorem}{Theorem}[section]
\newtheorem{lemma}[theorem]{Lemma}

\newtheorem{corollary}[theorem]{Corollary}
\newtheorem{proposition}[theorem]{Proposition}

\theoremstyle{definition}
\newtheorem{definition}[theorem]{Definition}

\theoremstyle{remark}
\newtheorem{remark}[theorem]{Remark}
\newtheorem {question}[theorem]{Question}

\numberwithin{equation}{section}

\usepackage{hyperref}
\usepackage{cleveref}

\renewcommand{\AA}{\mathbb{A}}

\newcommand{\NN}{\mathbb{N}}

\renewcommand{\L}{\mathcal{L}}

\newcommand{\D}{\mathcal{D}}

\newcommand{\C}{\mathcal{C}}

\newcommand{\V}{\mathcal{V}}
\newcommand{\U}{\mathcal{U}}

\newcommand{\W}{\mathcal{W}}

\renewcommand{\P}{\mathcal{P}}
\renewcommand{\O}{\mathcal{O}}

\newcommand{\RCA}{\mathsf{RCA}}
\newcommand{\ACA}{\mathsf{ACA}}
\newcommand{\WKL}{\mathsf{WKL}}

\newcommand{\ATR}{\mathsf{ATR}}
\newcommand{\PIOOCA}{\Pi^1_1\mbox{-}\mathsf{CA}}

\newcommand{\ISig}{\mathsf{I}\Sigma^0}

\newcommand{\BSig}{\mathsf{B}\Sigma^0}

\newcommand{\HT}{\mathsf{HT}}

\newcommand{\CT}{\mathsf{CT}}

\usepackage{stmaryrd} 
\usepackage{xcolor} 

\newcommand{\dom}{\operatorname{dom}}
\usepackage{centernot} 

\usepackage{xifthen}

\newcommand{\Orb}{\mathsf{Orb}}
\newcommand{\flim}[1]{#1\mbox{-}\lim}
\newcommand{\FS}{\operatorname{FS}}
\newcommand{\FU}{\operatorname{FU}}

\newcommand{\VW}{\operatorname{V}}
\newcommand{\Wop}{\operatorname{W}}

\renewcommand{\setminus}{\smallsetminus}
\newcommand{\FUT}{\mathsf{FUT}}
\newcommand{\IFUT}{\mathsf{IFUT}}
\newcommand{\varp}{\star}
\newcommand{\FIN}{\operatorname{FIN}}
\newcommand{\fulim}{\flim{\operatorname{FU}^{\leq 2}}}

\title[The RM of Carlson's theorem for located words]{The reverse mathematics of\\Carlson's theorem for located words}

\author{Tristan Bompard}

\email{tristanbompard2.7@gmail.com}

\author{Lu Liu}
\address{School of Mathematics and Statistics, HNP-LAMA,
Central South University,
City Changsha, Hunan Province,
China. 410083}
\email{g.jiayi.liu@gmail.com}

\author{Ludovic Patey}

\address{CNRS, Équipe de Logique\\Universit\'e de Paris\\ Paris, FRANCE}

\email{ludovic.patey@computability.fr}

\urladdr{http:/ludovicpatey.com}

\thanks{}

\begin{document}

\begin{abstract}
In this article, we give two proofs of Carlson's theorem for located words in~$\ACA^+_0$.
The first proof is purely combinatorial, in the style of Towsner's proof of Hindman's theorem. The second uses topological dynamics to show that an iterated version of Hindman's theorem for bounded sums implies Carlson's theorem for located words. 
\end{abstract}

\maketitle

\section{Introduction}

We study the metamathematics of a partition theorem for words due to Carlson~\cite{carlson1988some} from the viewpoint of reverse mathematics. Also our motivation is foundational, the metamathematical study of theorems in combinatorics usually consists in seeking for combinatorially simpler proofs of existing theorems. This is in particular the case of this article, where we give two new proofs of Carlson's theorem: a direct purely combinatorial one in the style of Towsner~\cite{towsner2012simple}, and another from the Finite Union theorem using the tools of topological dynamics. This article might therefore be of interest to both combinatoricians who can ignore the metamathematical considerations, and logicians who wonder about the optimal axioms to prove Carlson's theorem.

\subsection{Reverse mathematics}
This is a foundational program started in 1974 by Harvey Friedman, whose goal is to study the optimal axioms for proving ordinary theorems. It uses the framework of subsystems of second-order arithmetics, with a base theory, $\RCA_0$, capturing \emph{computable mathematics}. The early study of reverse mathematics has shown the emergence of four systems of axioms, namely, $\WKL_0$, $\ACA_0$, $\ATR_0$ and $\PIOOCA_0$, listed in increasing order in terms of logical strength, such that most theorems are either provable in~$\RCA_0$ (hence computably true), or equivalent modulo~$\RCA_0$ to one of these four systems. This observation is known as the \emph{Big Five phenomenon}. See Simpson~\cite{simpson2009subsystems} for a reference book on the early reverse mathematics.

The study of combinatorial theorems, especially coming from Ramsey's theory, has relativized the Big Five phenomenon. Ramsey's theorem for pairs is the most famous example of theorem which fails this observation in a strong sense: its logical strength is strictly in between $\RCA_0$ and $\ACA_0$ (see Specker~\cite{specker1971ramseys}, Jockusch~\cite{jockusch1972ramseys} and Seetapun and Slaman~\cite{seetapun1995strength}), and incomparable with~$\WKL_0$ (see Jockusch~\cite{jockusch1972ramseys} and Liu~\cite{liu2012rt2_2}). Combinatorial theorems are notoriously hard to study in reverse mathematics, and often require to find new elementary proofs of existing theorems. This is in particular the case of Hindman's theorem, whose logical strength is still an active study of research. See Hirschfeldt~\cite{hirschfeldt2015slicing} or Dzhafarov and Mummert~\cite{dzhafarov2022reverse} for an introduction to the  reverse mathematics of combinatorial theorems.

\subsection{Hindman's theorem}

Hindman's theorem~\cite{hindman1974finite} is a partition theorem about the integers.
Given a set~$X \subseteq \NN$, we let $\FS(X)$ be the set of non-empty finite sums of distinct elements from~$X$, that is,
$$
\FS(X) = \{ \Sigma F : F \subseteq X \wedge 0 < |F| < \infty \}
$$

A \emph{finite coloring} of a set~$X$ is a function of the form $X \to C$, where~$C$ is a finite set, identified as a set of colors.
Given a finite coloring~$f : X \to C$, we say that a subset~$Y \subseteq X$ is \emph{$f$-homogeneous} if every element in~$Y$ is given the same color by~$f$.

\begin{theorem}[Hindman]\label[theorem]{thm:hindman}
For every finite coloring $f : \NN \to C$, there is an infinite set~$Y \subseteq \NN$ such that
$\FS(Y)$ is $f$-homogeneous.
\end{theorem}

We shall refer to the previous theorem as Hindman's theorem ($\HT$).
There exist multiple proofs of Hindman's theorem. The original proof from Hindman~\cite{hindman1974finite}, a short proof from Baumgartner~\cite{baumgartner1974short}, an ultrafilter proof from Galvin and Glazer (see Hindman and Strauss~\cite{hindman2012algebra}), a proof using topological dynamics by Furstenberg and Weiss~\cite{furstenberg1978topological}, and a simple proof by Towsner~\cite{towsner2012simple}.
The two first proofs were analyzed in reverse mathematics by Blass, Hirst, and Simpson~\cite{blass1987logical}. They showed that the original proof from Hindman holds in~$\ACA^+_0$, while Baumgartner's proof can be formalized in the much stronger system~$\Pi^1_2\mbox{-}\mathsf{TI}_0$. The analysis of the proofs from Galvin and Glazer and from Furstenberg and Weiss are more tricky, since they use third-order objects : ultrafilters in the former case, and Ellis enveloping semigroups in the latter case. Montalban and Shore~\cite{montalban2018conservativity} studied the ultrafilter proof using the tools of conservativity, and proved that the existence of an idempotent ultrafilter is a conservative extension of $\ACA_0$ augmented with an iterated version of Hindman's theorem. Kreuzer~\cite{kreuzer2013minimal} studied the existence of idempotent ultrafilters from the viewpoint of higher-order reverse mathematics, and proved that iterated Hindman's theorem is equivalent to the Auslander-Ellis theorem in topological dynamics. Lastly, the proof from Towsner holds in~$\ACA^+_0$. The exact reverse mathematical strength of Hindman's theorem remains one of the biggest open questions in reverse mathematics.

Hindman's theorem is a typical example where the quest for optimal axioms motivates the search for new elementary proofs of existing theorems.
This motivated for example the purely combinatorial proof by Towsner~\cite{towsner2012simple}, which isarguably elementary as it is obtain by combining and iterating only very simple constructions. On the other hand, reverse mathematics is sensitive to the operation of iterating constructions. Iterations are very elementary from a mathematical perspective, but this yields sometimes computationally complex objects which then require strong existence axioms.

Hindman's theorem is equivalent to a partition theorem about finite sets, known as the Finite Union theorem.
Let~$\P_f(\NN)$ be the set of all non-empty finite subsets of~$\NN$.
A subset~$X \subseteq \P_f(\NN)$ is a \emph{block sequence} if for every~$E, F \in X$, either $\max E < \min F$, or $\max F < \min E$. Given a set~$X \subseteq \P_f(\NN)$,
we write $\FU(X)$ for the set of all non-empty finite unions of elements from~$X$, that is,
$$
\FU(X) = \{ \cup F : F \subseteq X \wedge 0 < |F| < \infty \}
$$
\begin{theorem}[Finite Union]\label[theorem]{thm:finite-union}
For every finite coloring $f : \P_f(\NN) \to C$, there is an infinite block sequence~$Y \subseteq \P_f(\NN)$ such that
$\FU(Y)$ is $f$-homogeneous.
\end{theorem}

The Finite Union theorem ($\FUT$) is known to be equivalent to Hindman's theorem over~$\RCA_0$, through the one-to-one correspondence which to~$n \in \NN$ associates the finite set~$F \in \P_f(\NN)$ such that $\sum_{i \in F} 2^i = n$. The Finite Union theorem is more convenient than Hindman's theorem from many viewpoints, as the set~$\FU(Y)$ is closed under finite unions, while the set~$\FS(Y)$ is not closed under finite sums.

\subsection{Carlson's theorem for words}

Carlson's theorem for words~\cite{carlson1988some} is a unifying theorem generalizing both Hindman's theorem and the Hales-Jewett theorem. It comes from a long line of partition theorems about variable words, among which the Hales-Jewett theorem~\cite{hales1963regularity}, the Graham-Rothschild theorem~\cite{graham1971ramseys} and the Carlson-Simpson lemma~\cite{carlson1984dual}.

\begin{definition}
A \emph{variable word} over an alphabet~$A$ is a finite sequence $w$ over~$A \cup \{\star\}$, where $\star$ is a variable which occurs at least once in~$w$. Given a variable word~$w$ and a letter~$a \in A$,
we write $w[a]$ for the string of length~$|w|$ where all the occurrences of~$\star$ are replaced by~$a$.
\end{definition}

We write $\Wop(A)$ and $\VW(A)$ for the sets of all words and variable words over~$A$, respectively.
The most basic -- and arguably the most important -- theorem about variable words is the Hales-Jewett theorem. It plays the same role on the semigroup of words as the pigeonhole principle does in the semigroup of integers.

\begin{theorem}[Hales-Jewett]
Fix a finite alphabet~$A$.
For every finite coloring $f : \Wop(A) \to C$, there is a variable word~$w$ such that $\{ w[a] : a \in A \}$ is $f$-homogeneous.
\end{theorem}

The Hales-Jewett theorem was generalized to arbitrary finite dimensions by Graham and Rothschild, and to infinite dimensions by Carlson and Simpson.
Carlson's theorem for words generalizes both the Carlson-Simpson lemma and Hindman's theorem as follows.

\begin{definition}
Given a finite or infinite sequence of variable words $(w_n)_{n < \ell}$ with $\ell \in \NN \cup \{\NN\}$,
a word~$u \in \Wop(A \cup \{\varp\})$ is \emph{extracted}\footnote{The terminology comes from \cite{dodos2016ramsey}.} from $(w_n)_{n < \ell}$ if there exists a finite sequence of indices $j_0 < \dots < j_{n-1}$ and a finite sequence of letters $a_0, \dots, a_{n-1} \in A \cup \{\varp\}$ such that
$$u = w_{j_0}[a_0] w_{j_1}[a_1] \dots w_{j_{n-1}}[a_{n-1}]$$
We write $\langle(w_n)_{n < \ell}\rangle_A$ and $\langle(w_n)_{n < \ell}\rangle_{A \varp}$ for the set of all words and variable words extracted from $(w_n)_{n < \ell}$, respectively.
\end{definition}

\begin{theorem}[Carlson for words]\label[theorem]{thm:carlson-theorem}
For every finite coloring~$f : \Wop(A) \to C$, there is an infinite sequence $(w_n)_{n \in \NN}$ of variable words such that $\langle(w_n)_{n \in \NN}\rangle_A$ is $f$-homogeneous.
\end{theorem}

Carlson's theorem for words implies Hindman's theorem by considering colorings which depend only on the length of the words.
Indeed, for every~$f : \NN \to C$, letting $g : \Wop(A) \to C$ be defined by $g(u) = f(|u|)$, then for any infinite $g$-homogeneous sequence of variable words~$(w_n)_{n \in \NN}$, letting~$Y = \{ |w_n| = n \in \NN \}$, the set~$\FS(Y)$ is $f$-homogeneous.

\begin{remark}
Carlson~\cite{carlson1988some} actually proved a stronger theorem about finite colorings of variable words : For every finite coloring~$f : \VW(A) \to C$, there is an infinite sequence $(w_n)_{n \in \NN}$ of variable words such that $\langle(w_n)_{n \in \NN}\rangle_{A \varp}$ is $f$-homogeneous. This statement, that we shall refer as Carlson's theorem for variable words, implies its version for words. The only known proofs of Carlson's theorem for variable words involve ultrafilters~\cite{carlson1988some} or the Ellis envelopping semigroup~\cite{furstenberg1989idempotents}.
\end{remark}

While Carlson's theorem for variable words is only known to admit an ultrafilter proof and a topological dynamical proof, Karagiannis~\cite{karagiannis2013combinatorial} gave a Baumgartner-style proof of Carlson's theorem for words. Despite their combinatorial simplicity, Baumgartner-style proofs involve strong set existence axioms from a metamethematical viewpoint. In this article, we give an alternative proof of Carlson's theorem for words in the style of Towsner~\cite{towsner2012simple}, which yields the same known upper bound as Hindman's theorem, namely, $\ACA^{+}_0$.

Carlson's theorem for words admits a formulation in terms of located words, which plays the same role as the Finite Union theorem for Hindman's theorem.

\begin{definition} Fix a finite alphabet.
\begin{enumerate}
	\item A \emph{located word} over~$A$ is a function from a finite nonempty subset of~$\NN$ into~$A$. Let~$\FIN_A$ be the collection of all located words over~$A$.
	\item A \emph{located variable word} over~$A$ is a finite partial function from~$\NN$ into~$A \cup \{\varp\}$, that takes the value~$\varp$ at least once. Let~$\FIN_{A \varp}$ be the collection of all located variable words.
\end{enumerate}
\end{definition}

Note that we have $\FIN_{A \cup \{\varp\}} = \FIN_A \sqcup \FIN_{A \varp}$.
Given a located variable word~$p \in \FIN_{A \varp}$ and a letter~$a \in A$, we write $p[a]$ for the located words with $\dom(p[a]) = \dom p$ and
$p[a](x) = a$ if $x = \star$, and $p[a](x) = p(x)$ otherwise. We also let $p[\varp] = p$.
We equip the collections $\FIN_A$ and $\FIN_{A \varp}$ with a partial ordering defined as
$$
p < q \mbox{ if } \max \dom p < \min \dom q
$$

\begin{definition}
A \emph{block sequence} is a totally ordered set~$X \subseteq \FIN_{A \varp}$.
Given a block sequence~$X$, we let
\begin{enumerate}
	\item $[X]_A = \{ p_0[a_0] \cup \dots \cup p_k[a_k] \in \FIN_A : p_0, \dots, p_k \in \FIN_{A \varp}, a_0, \dots, a_k \in A \}$
	\item $[X]_{A\varp} =\{ p_0[a_0] \cup \dots \cup p_k[a_k] \in \FIN_{A\varp} : p_0, \dots, p_k \in \FIN_{A \varp}, a_0, \dots, a_k \in A \cup \{\varp\} \}$
\end{enumerate}
\end{definition}

In other words, $[X]_{A\varp} = [X]_{A \cup \{\varp\}} \cap \FIN_{A \varp}$.

\begin{theorem}[Carlson for located words]\label[theorem]{thm:carlson-theorem}
For every finite coloring~$f : \FIN_A \to C$, there is an infinite block sequence $X \subseteq \FIN_{A \varp}$ such that $[X]_A$ is $f$-homogeneous.
\end{theorem}

Carlson's theorem for located words implies its version for words by collapsing the domain of the partial functions to obtain an initial segment of~$\NN$.
This implication will be formally proven in~\Cref{prop:equiv-carlson-versions}. Like the Finite Union theorem, Carlson's theorem for located words is more convenient to manipulate since~$[X]_A$ is closed under finite unions.

\subsection{Organization of the paper}

In \Cref{sect:versions-carlson}, we discuss various versions of Carlson's theorem for words and justify  the use of located words by the ability to iterate the theorem. In \Cref{sect:towsner-ct}, we give a Towsner-style proof of Carlson's theorem for located words. This proof can be formalized in~$\ACA^{+}_0$. In \Cref{sect:topo-dynamics}, we develop some basics of topological dynamics for located words, as a preparation for \Cref{sect:aet}. In \Cref{sect:aet} we state two versions of the Auslander-Ellis theorem for located words, and use them to give an alternative proof of Carlson's theorem for located words. We also prove that AET for located words follows from an iterated version of the Finite Union theorem. Since the latter theorem holds in~$\ACA^{+}_0$, this yields a second proof of Carlson's theorem for located words in~$\ACA^{+}_0$. 
    Lasst, in \Cref{sect:open-questions}, we state a few remaining open questions.

\section{Versions of Carlson's theorem for words}\label[section]{sect:versions-carlson}

One desirable feature of a partition theorem is the ability to iterate it, to obtain a simultaneous solution to multiple instances. This can be achieve whenever the combinatorial space representing the solution is isomorphic to the original space. We then obtain a stronger statement, saying that any partition of a specific combinatorial space admits a combinatorial subspace within one of the parts. In this section, we study the corresponding strengthenings for the Hindman's theorem, Carlson's theorem for words, and justify why the Finite Union theorem and Carlson's theorem for located words is more convenient in terms of iterations. In what follows, we consider only partial semigroups where the product is defined only on distinct elements.

\subsection{Strong Finite Union theorem and Hindman's theorem}
In this section, we shall see that both the Finite Union theorem and Hindman's theorem are equivalent to their strong version over~$\RCA_0$. However, there exists a natural semigroup isomorphism between the Finite Union theorem, its strong version and the strong version of Hindman's theorem, while the equivalence with Hindman's theorem requires some extra work.

The Finite Union theorem admits a natural iterable version thanks to the closure property of its combinatorial space.

\begin{theorem}[Strong Finite Union]
For every infinite block sequence~$X \subseteq \P_f(\NN)$ and every finite coloring $f : \FU(X) \to C$,
there is an infinite block sequence~$Y \subseteq \FU(X)$ such that $\FU(Y)$ is $f$-homogeneous.
\end{theorem}

The strong Finite Union theorem is an immediate consequence of the Finite Union theorem since for every infinite block sequence~$X = \{ F_0 < F_1 < \dots \} \subseteq \P_f(\NN)$, there is a natural isomorphism $\iota$ from $(\P_f(\NN), \cup)$ to $(\FU(X), \cup)$ defined by $\iota(E) = \bigcup_{n \in E} F_n$.

\begin{lemma}[$\RCA_0$]
The Finite Union theorem implies its strong version.	
\end{lemma}
\begin{proof}
Given $X = \{ F_0 < F_1 < \dots \}$ and $f : \FU(X) \to C$, let~$g : \P_f(\NN) \to C$ be defined by $g(E) = f(\iota(E))$. By the Finite Union theorem, there is an infinite block sequence~$Z \subseteq \P_f(\NN)$ such that $\FU(Z)$ is $g$-homogeneous. Let~$Y = \{ \iota(E) : E \in Z \}$. The set~$Y \subseteq \FU(X)$ is an infinite block sequence and $\FU(Y) = \{ \iota(E) : E \in \FU(Z) \}$, so $\FU(Y)$ is $f$-homogeneous.
\end{proof}

It follows that the Finite Union theorem is equivalent to its strong version over~$\RCA_0$. The strong version of Hindman's theorem requires more work, as given an infinite set~$X \subseteq \NN$ and an infinite set~$Y \subseteq \FS(X)$, $\FS(Y)$ is not necessarily a subset of~$\FS(X)$.

\begin{theorem}[Strong Hindman]
For every infinite set~$X \subseteq \NN$ and every finite coloring $f : \FS(X) \to C$,
there is an infinite set~$Y$ such that $\FS(Y) \subseteq \FS(X)$ and $\FS(Y)$ is $f$-homogeneous.
\end{theorem}

Unlike the Finite Union theorem, there is no natural isomorphism from $(\NN, +)$ to $(\FS(X), +)$, unless it satisfies some extra conditions.

\begin{definition}[Carlucci et al.~\cite{carlucci2020new}]
Given $n = 2^{n_0} + \dots + 2^{n_{\ell-1}}$ with $n_0 < \dots < n_{\ell-1}$, we let~$\lambda(n) = n_0$, $\mu(n) = n_{\ell-1}$. An infinite set~$A = \{ a_0 < a_1 < \dots \}$ is \emph{2-apart} if for every~$n$, $\mu(a_n) < \lambda(a_{n+1})$.
\end{definition}

The notion of 2-apartness for sets of integers is equivalent to the notion of block sequence for sets of finite sets. Suppose $X \subseteq \NN$ is infinite and 2-apart. Then there is an isomorphism $\iota$ from $(\NN, +)$ to $(\FS(X), +)$ defined by $\iota(\sum_{i \in E} 2^i) =  \sum_{i \in E} a_i$, where~$X = \{a_0 < a_1 < \dots \}$. The notion of 2-apartness is not preserved by subsets of finite sums in general: there exist sets $X, Y \subseteq \NN$ such that $X$ is 2-apart, $\FS(Y) \subseteq \FS(X)$ but $Y$ is not 2-apart. The following stronger relation preserves 2-apartness.

\begin{definition}
Given two infinite sets~$X, Y \subseteq \NN$, we write $Y \leq_{\FS} X$ if there is an infinite block sequence~$H \subseteq \P_f(X)$ such that $Y = \{ \Sigma F : F \in H \}$.
\end{definition}

Note that if~$Y \leq_{\FS} X$, then $\FS(Y) \subseteq \FS(X)$. Moreover, if $X$ is 2-apart, then so is~$Y$. The following lemma was proven by Hindman~\cite{hindman1974finite} in his original paper.

\begin{lemma}[Hindman, $\RCA_0$]\label[lemma]{lem:normalization-ip-set}
For every infinite set~$X \subseteq \NN$, there is an infinite 2-apart set~$Y \leq_{\FS} X$.
\end{lemma}
\begin{proof}
We first prove by induction over~$k \in \NN$ that for every infinite set~$X \subseteq \NN$,
there is some~$n \in \FS(X)$ such that $\lambda(n) \geq k$.
The case $k = 0$ is trivial by taking $n = \min X$.
Suppose by induction hypothesis it holds for~$k$, but not for~$k+1$.
Then for every finite set~$F \subseteq X$, there is some~$n \in \FS(X \setminus F)$ such that $\lambda(n) = k$.
Let $F_0, F_1$ be two non-empty subsets of~$X$ such that $\max F_0 < \min F_1$ and $\lambda(\sum F_i) = k$.
Then $\lambda(\sum F_0 \cup F_1) > k$, contradicting our hypothesis. This proves our claim.

Using the claim, we construct an infinite sequence $F_0, F_1, \dots$ of finite non-empty subsets of~$X$ such that $\max F_n < \min F_{n+1}$ and $\lambda(\sum F_{n+1}) > \lambda(\sum F_n)$.
First, $F_0 = \{ \min X \}$. Assume $F_0, \dots, F_n$ are defined.
By our claim, since $X \setminus \{0, 1, \dots, \max F_n\}$ is infinite, there is a finite set
$F_{n+1} \subseteq X \setminus \{0, 1, \dots, \max F_n\}$ such that $\lambda(\sum F_{n+1}) > \lambda(\sum F_n)$.
By construction, the set $Y = \{ \sum F_n : n \in \NN \}$ is 2-apart and $B \leq_{\FS} A$.
\end{proof}

We are now ready to prove strong Hindman's theorem from Hindman's theorem.

\begin{lemma}[$\RCA_0$]
Hindman's theorem implies its strong version.	
\end{lemma}
\begin{proof}
Given $X$ and $f : \FS(X) \to C$, by \Cref{lem:normalization-ip-set} there is an infinite 2-apart set~$X_0 \leq_{\FS} X$. Let~$\iota$ be the canonical isomorphism from $(\NN, +)$ to $(\FS(X_0), +)$.
Let~$g : \FS(\NN) \to C$ be defined by $g(n) = f(\iota(n))$. By Hindman's theorem, there is an infinite set~$Z \subseteq \NN$ such that $\FS(Z)$ is $g$-homogeneous. Let~$Y = \{ \iota(n) : n \in Z \}$. The set~$Y$ satisfies $\FS(Y) \subseteq \FS(X)$, and $\FS(Y) = \{ \iota(n) : n \in \FS(Z) \}$, so $\FS(Y)$ is $f$-homogeneous.
\end{proof}

Thanks to the equivalence between Hindman's theorem and the Finite Union theorem, we have the following equivalence.

\begin{proposition}[Folklore]\label[proposition]{prop:equiv-hindman-fut}
The following statements are equivalent over~$\RCA_0$:
\begin{enumerate}
	\item Hindman's theorem
	\item Strong Hindman's theorem
	\item Finite Union theorem
	\item Strong Finite Union theorem
\end{enumerate}
\end{proposition}
\begin{proof}
We have already proven the equivalences~$(1) \leftrightarrow (2)$ and $(3) \leftrightarrow (4)$.
The proof of $(3) \rightarrow (1)$ is immediate, using the isomorphism~$\iota$ from $(\P_f(\NN), \cup)$ to $(\NN, +)$ defined by $\iota(E) = \sum_{i \in E} 2^i$. Last, $(2) \rightarrow (3)$ is immediate, using the isomorphism $\iota$ from $(\FS(X), +)$ to $(\P_f(\NN), \cup)$ defined by  $\iota(\sum_{i \in E} 2^i)  = E$.
\end{proof}

\begin{remark}
The previous considerations show that there exists are direct correspondance between the Finite Union theorem, its strong version, and the strong version of Hindman's theorem. On the other hand, the proof of any of these theorems from Hindman's theorem requires a preliminary lemma, namely, \Cref{lem:normalization-ip-set}. In particular, the Finite Union theorem is naturally equivalent to its strong version, while the equivalence between Hindman's theorem and its strong version is arguably less natural.
\end{remark}

\subsection{Strong Carlson's theorem for located words and words}
Carlson's theorem for located words and for words both admit a strong version. As before, Carlson's theorem for located words is naturally equivalent to its strong version and to the strong version of Carlson's theorem for words. However, it remains open whether Carlson's theorem for words implies its strong version over~$\RCA_0$.

\begin{theorem}[Strong Carlson for located words]
Let $X \subseteq \FIN_{A \varp}$ be an infinite block sequence and let $f : [X]_{A \varp} \to C$ be a coloring.
There is an infinite block sequence~$Y \subseteq [X]_{A \varp}$ such that $[Y]_{A \varp}$ is $f$-homogeneous.
\end{theorem}

The proof of strong Carlson's theorem for located words from its weak version follows from the isomorphism~$\iota$ from $(\FIN_{A \cup \{\varp\}}, \cup)$ to $([X]_{A \cup \{\varp\}}, \cup)$ defined by~$\dom \iota(q)  = \bigsqcup_{n \in \dom q} \dom p_n$ and for~$m \in \dom q$ and $n \in \dom p_n$, $\iota(q)(n) = p_n(q(m))$, where~$X = \{ p_0 < p_1 < \dots \}$ is an infinite block sequence.

\begin{lemma}[$\RCA_0$]
Carlson's theorem for located words implies its strong version.	
\end{lemma}
\begin{proof}
Given $X = \{ p_0 < p_1 < \dots \}$ and $f : [X]_A \to C$, let~$g : \FIN_A \to C$ be defined by $g(p) = f(\iota(p))$. By Carlson's theorem for located words, there is an infinite block sequence~$Z \subseteq \FIN_{A \varp}$ such that $[Z]_A$ is $g$-homogeneous. Let~$Y = \{ \iota(p) : p \in Z \}$. The set~$Y \subseteq [X]_{A \varp}$ is an infinite block sequence and $[Y]_A = \{ \iota(p) : p \in [Z]_A \}$, so $[Y]_A$ is $f$-homogeneous.
\end{proof}

\begin{theorem}[Strong Carlson for words]
Let $(w_n)_{n \in \NN}$ be an infinite sequence of variable words over~$A$ and let $f : \langle (w_{n})_{n \in \NN}\rangle_A \to C$ be a finite coloring.
There is an infinite sequence of variable words $(u_n)_{n \in \NN}$ such that $\langle (u_n)_{n \in \NN} \rangle_A \subseteq \langle (w_n)_{n \in \NN}\rangle_A$ and $\langle(u_{n})_{n \in \NN}\rangle_A$ is $f$-homogeneous.
\end{theorem}

Strong Carlson's theorem for words follows from Carlson's theorem for located words, using the surjective morphism $\iota$ from $(\FIN_{A \cup \{\varp\}}, \cup)$ to $(\langle (w_n)_{n \in \NN}\rangle_{A \cup \{\varp\}}, \cdot)$ defined by collapsing the domains as follows: $\iota(p) = w_{n_0}[p(n_0)] \cdots w_{n_{k-1}}[p(n_{k-1})]$ where~$\dom p = \{n_0 < \dots < n_{k-1}\}$.

\begin{lemma}[$\RCA_0$]
Carlson's theorem for located words implies strong Carlson's theorem for words.
\end{lemma}
\begin{proof}
Fix $(w_n)_{n \in \NN}$ and~$f$. Let~$g : \FIN_A \to C$ be defined by~$g(p) = f(\iota(p))$.
By Carlson's theorem for located words, there is an infinite block sequence~$X = \{p_0 < p_1 < \dots \} \subseteq \FIN_{A\varp}$ such that $[X]_A$ is $g$-homogeneous. Let~$(u_n)_{n \in \NN}$ be defined by
$u_n = \iota(p_n)$. Then $\langle(u_{n})_{n \in \NN}\rangle_A = \{ \iota(p) : p \in [X]_A\}$, hence $\langle(u_{n})_{n \in \NN}\rangle_A$ is $f$-homogeneous.
\end{proof}

As for Hindman's theorem, the strong version of Carlson's theorem for words implies Carlson's theorem for located words in a natural way. Indeed, let~$(w_n)_{n \in \NN}$ be the sequence of variable words over~$A$ defined by~$w_n = \varp \cdots \varp$ with $|w_n| = 2^n$. Any element in $\langle (w_n)_{n \in \NN} \rangle_{A \cup \{\varp\}}$ can be uniquely written in the form $w_{n_0}[a_0] \cdots w_{n_{k-1}}[a_{k-1}]$, where~$n_0 < \dots < n_{k-1} \in \NN$ and $a_0, \dots, a_{k-1} \in A \cup \{\varp\}$. Thus, there is a natural isomorphism~$\iota$ from $(\langle (w_n)_{n \in \NN} \rangle_{A \cup \{\varp\}}, \cdot)$ to $(\FIN_{A \cup \{\varp\}}, \cup)$ defined by $\iota(w_{n_0}[a_0] \cdots w_{n_{k-1}}[a_{k-1}]) = p$ such that $\dom(p) = \{n_0 < \dots < n_{k-1}\}$ and $p(n_i) = a_i$.

\begin{lemma}[$\RCA_0$]
Strong Carlson's theorem for words implies Carlson's theorem for located words.
\end{lemma}
\begin{proof}
Let~$(w_n)_{n \in \NN}$ be the sequence of variable words over~$A$ defined as above and let~$\iota$ be the corresponding isomorphism. Given a finite coloring~$f : \FIN_A \to C$, let $g : \langle (w_n)_{n \in \NN} \rangle_A \to C$ be defined by $g(w) = f(\iota(w))$. By strong Carlson's theorem for words,
there is an infinite sequence of variable words $(u_n)_{n \in \NN}$ such that $\langle (u_n)_{n \in \NN} \rangle_A \subseteq \langle (w_n)_{n \in \NN}\rangle_A$ and $\langle(u_{n})_{n \in \NN}\rangle_A$ is $g$-homogeneous. Let~$Y = \{ \iota(u_n) : n \in \NN \}$. Then $Y$ is an infinite block sequence and $[Y]_A = \{ \iota(w) : w \in \langle(u_{n})_{n \in \NN}\rangle_A \}$, hence $[Y]_A$ is $f$-homogeneous.
\end{proof}

Put altogether, we obtain the following equivalences over~$\RCA_0$.



\begin{proposition}\label[proposition]{prop:equiv-carlson-versions}
The following statements are equivalent over~$\RCA_0$:
\begin{enumerate}
	\item Carlson's theorem for located words
	\item Strong Carlson's theorem for located words
	\item Strong Carlson's theorem for words
\end{enumerate}
\end{proposition}
\begin{proof}
We have proven $(1) \rightarrow (3)$, $(1) \rightarrow (2)$ and $(3) \rightarrow (1)$. Last, $(2) \rightarrow (1)$ since strong Carlson's theorem for located words is a generalization of Carlson's theorem for words.
\end{proof}

As mentioned at the start of the section, we leave the following question open.

\begin{question}
Does $\CT$ for words imply $\CT$ for located words over~$\RCA_0$?	
\end{question}

\section{Towsner-style proof of Carlson's theorem for located words}\label[section]{sect:towsner-ct}

As mentioned in the introduction, Hindman's theorem admits multiple proofs: an ultrafilter one from Galvin and Glazer (see Hindman and Strauss~\cite{hindman2012algebra}), a simple one from Baumgartner~\cite{baumgartner1974short} and a proof by Towsner~\cite{towsner2012simple}. The latter one yields the best known upper bound in terms of reverse mathematics of Hindman's theorem.

Carlson's theorem for words original proof involves ultrafilters, and Karagiannis~\cite{karagiannis2013combinatorial} later gave a Baumgartner-style proof of it.
In this section, we give a Towsner-style proof of Carlson's theorem for located words. Its reverse mathematical analysis shows that Carlson's theorem for located words holds in~$\ACA^+_0$. We shall refine this analysis in \Cref{sect:aet} by showing that Carlson's theorem for located words follows from an iterated version of the Finite Union theorem, which is also known to hold in~$\ACA^+_0$.

We will be using the following theorem as a blackbox.
\begin{theorem}[Hales-Jewett, $\RCA_0$]\label[theorem]{thm:hales-jewett}
Let~$A$ be a finite alphabet and $f : \FIN_A \to C$ be a finite coloring.
Then there is a located variable word~$p \in \FIN_{A \varp}$ such that $[p]_A$ is $f$-homogeneous.
\end{theorem}

The previous theorem does not require $\BSig_2$ since it admits a finite combinatorial version which requires less induction.

\begin{definition}
Let~$X \subseteq \FIN_{A \varp}$ be an infinite block sequence and let~$f : [X]_A \to C$ be a coloring.
A block sequence~$Y \subseteq [X]_{A \varp}$ is
\begin{enumerate}
	\item \emph{weakly $f$-thin for color~$i \in C$} if for every~$p \in [Y]_{A \varp}$, there is some~$a \in A$ such that $f(p[a]) \neq i$
	\item \emph{$f$-thin for color~$i \in C$} if for every~$p \in [Y]_A$, $f(p) \neq i$.
\end{enumerate}
\end{definition}

\begin{lemma}[$\RCA_0$]\label[lemma]{lem:vw-large-to-homogeneous}
Let $X \subseteq \FIN_{A \varp}$ be an infinite block sequence and let $f : [X]_A \to C$ be a coloring such that $X$ is weakly $f$-thin for some color~$i \in C$. Then there is an infinite block sequence $Y \subseteq [X]_{A \varp}$ which is $f$-thin for color~$i$.
\end{lemma}
\begin{proof}
We build a sequence of located variable words $q_0 < q_1 < \dots$ by induction.
For every infinite block sequence~$X$, let $\iota_X : \FIN_{A \cup \{\varp\}} \to [X]_{A \cup \{\varp\}}$ be the canonical isomorphism.

Let~$g : \FIN_A \to C$ be defined by~$g(p) = f(\iota_X(p))$.
By the Hales-Jewett theorem (\Cref{thm:hales-jewett}), there is a located variable word~$p$ over~$A$ such that $[p]_A$ is $g$-homogeneous. Let $q_0 = \iota_X(p)$
Then $[q_0]_A$ is $f$-homogeneous. Since~$q_0 \in X$, then there is some~$a \in A$ such that $f(q_0[x]) \neq i$, hence $\{q_0\}$ is $f$-thin for color~$i$.

Assume $q_0 < \dots < q_{n-1}$ are located variable words over~$A$ such that $F = \{q_0, \dots, q_{n-1}\}$ is $f$-thin for color~$i$.
Let~$g : \FIN_A \to C^{|[F]_A|+1}$ be defined by
$$
g(p) = \langle f(q \cup \iota_{X - F}(p)) : q \in [F]_A \cup \{\emptyset\} \rangle
$$
By the Hales-Jewett theorem (\Cref{thm:hales-jewett}), there is a located variable word~$p$ over~$A$ such that $[p]_A$ is $g$-homogeneous. Let $q_n = \iota_{X - F}(p)$.
Fix some~$q \in [F]_A \cup \{\emptyset\}$. Then $[q \cup q_n]_A$ is $f$-homogeneous. Since $q \cup q_n \in X$, then there is some~$a \in A$ such that $f(q \cup q_n[a]) \neq i$, thus $[q \cup q_n]_A$ is $f$-thin for color~$i$. Since it is the case for every~$q \in [F]_A \cup \{\emptyset\}$, then $F \cup \{q_n\}$ is $f$-thin for color~$i$.
\end{proof}

The following definitions are direct adaptations of Towsner's notions of half matches and full matches to variable words (see \cite{towsner2012simple}).

\begin{definition}
Let $X \subseteq \FIN_{A \varp}$ be an infinite block sequence and let $f : [X]_A \to C$ be a coloring. Let~$F \subseteq [X]_{A \varp}$ be a finite set and $Y \subseteq [X - F]_{A\varp}$ be an infinite block sequence. We say that
\begin{enumerate}
	\item $F$ \emph{half-matches $Y$ for color~$i \in C$} if for every~$q \in [Y]_A$ such that $f(q) = i$, there is some~$p \in F$ such that for every~$a \in A$, $f(p[a] \cup q) = i$.
	\item $F$ \emph{half-matches} $Y$ if it half-matches $Y$ for every color~$i \in C$
	\item $F$ \emph{full-matches} $Y$ if for every $q \in [Y]_A$, there is some~$p \in F$ such that for every~$a \in A$, $f(p[a]) = f(p[a] \cup q) = f(q)$.
\end{enumerate}
\end{definition}

\begin{lemma}[$\RCA_0$]\label[lemma]{lem:carlson-half-case}
Let $X \subseteq \FIN_{A \varp}$ be an infinite block sequence, let $f : [X]_A \to C$ be a coloring, and $i \in C$ be a color. There is a finite set~$F \subseteq [X]_{A \varp}$ and an infinite block sequence $Y \subseteq [X - F]_{A \varp}$ such that $F$ half-matches $Y$ for color~$i$.
\end{lemma}
\begin{proof}
Suppose first that for every finite block sequence~$F \subseteq [X]_{A \varp}$, there is a located variable word~$q \in [X - F]_{A \varp}$ and some~$a \in A$ such that for every~$p \in [F]_{A\varp}$, there is some~$b \in A$ such that $f(p[b]\cup q[a]) \neq i$. Then build an infinite block sequence $\{ p_0 < p_1 < \dots \} \subseteq [X]_{A \varp}$ such that for every~$n$, there is some~$a_n \in A$ such that for every~$p \in [p_0, \dots, p_{n-1}]_{A \varp}$, there is some~$b \in A$ such that $f(p[b \cup ]p_n[a_n]) \neq i$. Let~$Y = \{q_0 < q_1 < \dots \}$ be the block sequence defined by $q_n = p_{2n} \cup p_{2n+1}[a_{2n+1}]$. Note that $Y \subseteq [X]_{A \varp}$.

We claim that $Y$ is weakly $f$-thin for color~$i$. Let $p \in [Y]_{A \varp}$.
Then $p = \bigcup_{n \in \dom q} q_n[q(n)]$ for some located variable word~$q \in \FIN_{A \varp}$.
Let~$m = \max \dom q$.
We have $q_m[q(m)] = p_{2m}[q(m)] \cup p_{2m+1}[a_{2m+1}]$.
Let~$r = \bigcup_{n \in \dom q \setminus \{m\}} q_n[q(n)]$. Then $p = r \cup  p_{2m}[q(m)] \cup p_{2m+1}[a_{2m+1}]$, with $r \in [p_0, \dots, p_{2m-1}]_{A \varp}$, so by definition of $\{ p_0 < p_1 < \dots \}$, there is some~$b \in A$ such that $f(r[b] \cup p_{2m}[q(m)] \cup p_{2m+1}[a_{2m+1}]) \neq i$. In other words, there is some~$b \in A$ such that $f(p[b]) \neq i$. Hence $Y$ is weakly $f$-thin for color~$i$. By \Cref{lem:vw-large-to-homogeneous}, there is an infinite block sequence~$Z \subseteq [Y]_{A \varp}$ which is $f$-thin for color~$i$.
In particular, $\min Z$ half-matches $Z - \{\min Z\}$ for color~$i$.

Suppose now that there is a finite block sequence~$F \subseteq [X]_{A \varp}$ such that for every located variable word~$q \in [X - F]_{A \varp}$ and every~$a \in A$, there is some~$p \in [F]_{A\varp}$ such that for every~$b \in A$, $f(p[b]\cup q[a]) = i$.
Then by assumption, $F$ half-matches $X - F$ for color~$i$.
\end{proof}

\begin{lemma}[$\RCA_0+\ISig_2$]\label[lemma]{lem:carlson-half}
Let $X \subseteq \FIN_{A \varp}$ be an infinite block sequence and let $f : [X]_A \to C$ be a coloring. There is a finite set~$F \subseteq [X]_{A \varp}$ and an infinite block sequence $Y \subseteq [X - F]_{A \varp}$ such that $F$ half-matches~$Y$.
\end{lemma}
\begin{proof}
Let~$C = \{i_0, i_1, \dots, i_{|C|-1}\}$. We build a finite sequence of pairs
$$(F_0, Y_0), \dots, (F_{|C|}, Y_{|C|})$$  inductively as follows:
Initially, $F_0 = \emptyset$ and $Y_0 = X$.
Assuming $F_s$ and $Y_s$ have been defined, by \Cref{lem:carlson-half-case}
there is a finite set $F_{s+1} \subseteq [X_s]_{A \varp}$ and an infinite block sequence $Y_{s+1} \subseteq [X_s - F_{s+1}]_{A \varp}$ such that
$F_{s+1}$ half-matches $Y_{s+1}$ for color~$i_s$.
Let~$F = \bigcup_{s \leq |C|} F_s$ and $Y = Y_{|C|}$.

We claim that $F$ half-matches $Y$.
Let~$q \in [Y]_A$. Let~$s$ be such that $f(q) = i_s$. Then since $F_{s+1}$ half-matches $Y_{s+1}$ for color~$i_s$ and $q \in [Y_{s+1}]_A$, there is some $p \in F_{s+1} \subseteq F$ such that for every~$a \in A$, $f(p[a] \cup q) = i_s = f(q)$.
\end{proof}

\begin{lemma}[$\ACA_0$]\label[lemma]{lem:carlson-full-case}
Let $X \subseteq \FIN_{A \varp}$ be an infinite block sequence and let $f : [X]_A \to C$ be a coloring. One of the following holds:
\begin{enumerate}
	\item[(1)] There is an infinite block sequence $Y \subseteq [X]_{A \varp}$ which is $f$-thin for some color~$i \in C$.
	\item[(2)] There is a finite set~$F \subseteq [X]_{A \varp}$ and an infinite block sequence $Y \subseteq [X - F]_{A \varp}$ such that $F$ full-matches~$Y$.
\end{enumerate}
\end{lemma}
\begin{proof}
Construct an infinite sequence of finite sets of located variable words $F_1, F_2, \dots$,
an infinite sequence of infinite block sequences $Y_0, Y_1, \dots$
and an infinite sequence of colorings $f_0, f_1, \dots$ as follows.

Initially, $Y_0 = X$ and $f_0 = f$.
Suppose $Y_s$ and $f_s$ have been defined.
By \Cref{lem:carlson-half}, there is a finite set~$F_{s+1} \subseteq [Y_s]_{A \varp}$ and an infinite block sequence $Y_{s+1} \subseteq [Y_s - F_{s+1}]_{A \varp}$ such that $F_{s+1}$ half-matches $Y_{s+1}$ for the coloring~$f_s$. Let~$f_{s+1}$ be defined on $[X_{s+1}]_A$ by letting $f_{s+1}(q) = \langle p, f_s(q) \rangle$, where $p \in F_{s+1}$ is such that for every~$a \in A$, $f_s(p[a] \cup q) = f_s(q)$.

Suppose first that there is some~$s$ such that for every~$q \in [Y_{s+1}]_A$, there is some~$p_1 \in F_1, \dots, p_s \in F_s$ and~$p \in [p_1, \dots, p_s]_{A \varp}$ such that for every~$a \in A$, $f(p[a]) = f(p[a] \cup q) = f(q)$. Then, letting $$F = \bigcup_{p_1 \in F_1, \dots, p_s \in F_s} [p_1, \dots, p_1]_{A \varp}$$
the set $F$ full-matches $Y_{s+1}$ and we are done.

Suppose now that for every~$s \in \NN$, there is some~$q_s \in [Y_{s+1}]_A$ such that for every~$p_1 \in F_1, \dots, p_s \in F_s$ and $p \in [p_1, \dots, p_s]_{A \varp}$, there is some~$a \in A$ such that if $f(p[a]\cup q_s) = f(q_s)$, then $f(p[a]) \neq f(q_s)$.
By the pigeonhole principle, there is a color~$i \in C$ and an infinite sequence of integers $s_0 < s_1 < \dots$ such that for every~$r \in \NN$, $f(q_{s_r}) = i$. By choice of the colorings~$f_0, f_1, \dots$ and the definitions of half-matches, for each such~$r$, we can find a sequence $p_1 \in F_1, \dots, p_{s_r} \in F_{s_r}$ such that for every~$p \in [p_1, \dots, p_{s_r}]_{A \varp}$ and every~$a \in A$, $f(p[a] \cup q_{s_r}) = f(q_{s_r}) = i$.
Then, by our supposition, for every~$p \in [p_1, \dots, p_{s_r}]_{A \varp}$, there is some~$a \in A$ such that $f(p[a]) \neq i$.

By weak K\"onig's lemma (which holds in~$\ACA_0$), there is an infinite block sequence~$Y \subseteq [X]_{A \varp}$ which is weakly $f$-thin for color~$i$.
By \Cref{lem:vw-large-to-homogeneous}, there is an infinite block sequence~$Z \subseteq [Y]_{A \varp}$ which is $f$-thin for color~$i$.
\end{proof}

\begin{lemma}[$\ACA_0$]\label[lemma]{lem:carlson-full}
Let $X \subseteq \FIN_{A \varp}$ be an infinite block sequence and let $f : [X]_A \to C$ be a coloring. There is a finite set~$F \subseteq [X]_{A \varp}$ and an infinite block sequence $Y \subseteq [X - F]_{A \varp}$ such that $F$ full-matches~$Y$.
\end{lemma}
\begin{proof}
By induction on~$|C|$. Applying \Cref{lem:carlson-full-case} with $|C| = 1$, the second case must hold and we are done. Suppose now that $|C| > 0$ and that the property holds for~$|C|-1$. Applying \Cref{lem:carlson-full-case}, either the second case holds, in which case we are done, or the first case holds, and we apply the induction hypothesis.
\end{proof}

\begin{theorem}[$\ACA^+_0$]
Let $X \subseteq \FIN_{A \varp}$ be an infinite block sequence and let $(f_n)_{n \in \NN}$ be a sequence of colorings $[X]_A \to C$,
There is an infinite block sequence $Y \subseteq [X]_{A \varp}$ such that for every~$n$, there is some finite set~$F \subseteq Y$ such that $[Y - F]_A$ is $f_n$-homogeneous.
\end{theorem}
\begin{proof}
As in \Cref{lem:carlson-full-case}, construct an infinite sequence of finite sets of located variable words $F_1, F_2, \dots$,
an infinite sequence of infinite block sequences $Y_0, Y_1, \dots$
and an infinite sequence of colorings $g_0, g_1, \dots$ as follows.

Initially, $Y_0 = X$ and $g_0 = f_0$.
Suppose $Y_s$ and $g_s$ have been defined.
By \Cref{lem:carlson-full}, there is a finite set~$F_{s+1} \subseteq [Y_s]_{A \varp}$ and an infinite block sequence $Y_{s+1} \subseteq [Y_s - F_{s+1}]_{A \varp}$ such that $F_{s+1}$ full-matches $Y°{s+1}$ for the coloring~$g_s$. Let~$g_{s+1}$ be defined on $[X_{s+1}]_A$ by letting $g_{s+1}(q) = \langle p, g_s(q), f_{s+1}(q) \rangle$, where $p \in F_{s+1}$ is such that for every~$a \in A$, $g_s(p[a]) = g_s(p[a] \cup q) = g_s(q)$.

For every~$s$, there is a sequence $p_1 \in F_1, \dots, p_s \in F_s$ such that for each~$t \leq s$, $[p_t, \dots, p_s]_A$ is $f_t$-homogeneous. By weak K\"onig's lemma, there is an infinite block sequence $Y \subseteq [X]_{A \varp}$ such that for every~$n$, there is some finite set~$F \subseteq Y$ such that $[Y - F]_A$ is $f_n$-homogeneous.
\end{proof}

\section{Topological dynamics in~$C^{\FIN_A}$}\label[section]{sect:topo-dynamics}

Topological dynamics studies recurrence phenomena in the context of compact topological spaces.
Furstenberg and Weiss~\cite{furstenberg1978topological} have shown that it was a very powerful tool to study combinatorial theorems. They gave in particular a proof of Hindman's theorem based on the Auslander-Ellis theorem, a theorem of topological dynamics informally stating that for every finite coloring~$f : \NN \to C$, there exists another coloring~$g : \NN \to C$ similar to~$f$, but which enjoys a strong recurrence property. The proof of the Auslander-Ellis theorem involved the notion of Ellis' enveloping semigroup, which is a third-order object related to the notion of ultrafilter.

Blass, Hirst and Simpson~\cite{blass1987logical} studied the work of Furstenberg and Weiss in the context of reverse mathematics, and got rid of the use of Ellis' enveloping semigroup by using iterated Hindman's theorem in the proof of the Auslander-Ellis theorem. Kreuzer~\cite{kreuzer2013minimal} proved the equivalence between iterated Hindman's theorem, the Auslander-Ellis theorem and the existence of an idempotent ultrafilter in higher-order reverse mathematics.

The original work of Furstenberg and Weiss~\cite{furstenberg1978topological} was on finite colorings of the semigroup $(\NN, +)$. Furstenberg and Katznelson~\cite{furstenberg1989idempotents} extended their work to the semigroup $(\W(A), \cdot)$ of words over a finite alphabet~$A$, and proved Carlson's theorem for words and variable words using topological dynamics, using Ellis' enveloping semigroup.

In this section, we build on the work of Furstenberg and Katznelson and develop some basics of topological dynamics for located words. The tools will be used in \Cref{sect:aet} to state and prove multiple versions of the Auslander-Ellis theorem from the iterated Finite Union theorem.

Consider $\FIN_A$ as a partial semigroup with the $\cup$ operation, defined only whenever~$p < q$.
Fix a finite set $C$ of colors.
$C^{\FIN_A}$ is the space of all mappings from~$\FIN_A$ to~$C$.
Given~$m \in \NN$, $n \in \NN \cup \{\infty\}$, we write~$\FIN_A(m, n)$ for the set of all located words~$p \in \FIN_A$ such that $m \leq \dom(p) < n$.
Given~$\ell \in \NN$ and $p \in \FIN_A(\ell, \infty)$, let~$S^\ell_p(f) : \FIN_A(0, \ell) \to C$
be defined by $S^\ell_p(f)(q) = f(q \cup p)$. Whenever~$p = \emptyset$, we simply write~$S^\ell(f)$ for~$S^\ell_\emptyset(f)$.
We say that a finite coloring~$h : \FIN_A(0, \ell) \to C$ is a \emph{factor} of $f : \FIN_A \to C$ if
there is some~$p \in \FIN_A$ such that $S^\ell_p(f) = h$.

%
%

\subsection{Recurrence}

Informally, a coloring is recurrent if any initial segment of it occurs arbitrarily far. The located words being non-linearly ordered, there exist multiple notions of recurrence for colorings of located words.

\begin{definition}
A coloring~$f : \FIN_A \to C$ is
\begin{enumerate}
	\item \emph{weakly recurrent} if for every~$\ell \in \NN$,
there is some located word~$p \in \FIN_A(\ell, \infty)$ such that
$S^\ell_p(f) = S^\ell(f)$.
	\item \emph{recurrent} if for every~$\ell \in \NN$,
there is some located variable word~$p \in \FIN_{A \varp}(\ell, \infty)$ such that for each~$a \in A$,
$S^\ell_{p[a]}(f) = S^\ell(f)$.
	\item \emph{uniformly recurrent} if for every~$\ell \in \NN$,
there is some~$m > \ell$ such that for every~$p \in \FIN_A(m, \infty)$, there is some~$q \in \FIN_A(\ell, m)$ for which $S^\ell(f) = S^\ell_{q \cup p}(f)$.
\end{enumerate}
\end{definition}

The terminology suggests that uniform recurrence is a stronger notion than recurrence, although it is not obvious from the definitions. The following lemma shows that it is the case, thanks to the Hales-Jewett theorem.

\begin{lemma}[$\RCA_0$]
If~$f$ is uniformly recurrent, then it is recurrent.	
\end{lemma}
\begin{proof}
Fix~$\ell \in \NN$. Let~$m > \ell$ be such that for every~$p \in \FIN_A(m, \infty)$, there is some~$q \in \FIN_A(\ell, m)$
for which $S^\ell(f) = S^\ell_{q \cdot p}(f)$.

Let~$\pi_m : \FIN_A \to \FIN_A(m, \infty)$ be the canonical bijection.
Let~$g : \FIN_A \to \FIN_A(\ell, m)$ be defined for every~$p \in \FIN_A(m, \infty)$ by letting $g(p)$ be the least $q \in \FIN_A(\ell, m)$ (in any fixed order) such that $S^\ell_{q \cup \pi_m(p)}(f) = S^\ell(f)$.
By the Hales-Jewett theorem (\Cref{thm:hales-jewett}), there is a located variable word $v \in \FIN_{A \varp}$ such that $[v]_A$ is $g$-homogeneous, for some color~$q \in \FIN_A(\ell, m)$. Let~$w = q \cup \pi_m(v)$. In particular, for every~$a \in A$, $g(v[a]) = q$, so $S^\ell_{w[a]}(f) = S^\ell_{q \cup \pi_m(v[a])}(f) = S^\ell(f)$.
\end{proof}

The following lemma shows that recurrent colorings are simple instances of Carlson's theorem for located words, in that they admit solutions computable in the instances.

\begin{lemma}[$\RCA_0$]\label[lemma]{lem:recurrent-carlson}
For every recurrent coloring~$f : \FIN_A \to C$, there is an infinite block sequence $X \subseteq \FIN_{A \varp}$ such that $[X]_A$ is $f$-homogeneous for color~$f(\emptyset)$.
\end{lemma}
\begin{proof}
We build a sequence~$p_0 < p_1 < \dots$ of located variable words inductively.
Since~$f$ is recurrent, there is some located variable word~$p_0 \in \FIN_{A \varp}$ such that for each~$a \in A$, $S^0_{p_0[a]}(f) = S^0(f)$.
In particular, $f(p_0[a]) = f(\emptyset)$ for every~$a \in A$.

Assume~$p_0 < \dots < p_n$ are located variable words such that $[p_0, \dots, p_n]_A$ is $f$-homogeneous for color~$f(\emptyset)$. Let~$\ell = 1 \max \dom p_n$.
Since~$f$ is recurrent, there is some located variable word~$p_{n+1} \in \FIN_{A \varp}(\ell, \emptyset)$ such that for each~$a \in A$, $S^\ell_{p_{n+1}[a]}(f) = S^\ell(f)$.
Let~$p \in [p_0, \dots, p_n]_A \cup \{\emptyset\}$. Note that $\max \dom p < \ell$ and that $f(p) = f(\emptyset)$. Then for every~$a \in A$,
$$f(p \cup p_{n+1}[a]) = S_{p_{n+1}[a]}(f)(p) = g(p) = f(\emptyset)$$
Thus, $[p_0, \dots, p_n, p_{n+1}]_A$ is $f$-homogeneous for color~$f(\emptyset)$.
\end{proof}


\subsection{Minimal subshifts}

In general, a dynamical system induces a notion of subshift, which is a class of colorings closed under application of the homeomorphism. In the case of colorings of located words, the situation is slightly more complex because of the partiality of the semigroup. We can however define some notion of subshift, and recover the standard properties saying that any minimal subshift contains only uniformly recurrent colorings.

\begin{definition}
A non-empty closed class~$\C \subseteq C^{\FIN_A}$ is a \emph{subshift} if for every~$f \in \C$,
every~$\ell \in \NN$ and $p \in \FIN_A(\ell, \infty)$,
there is some~$g \in \C$ such that $S^\ell(g) = S^\ell_p(f)$.
\end{definition}

A subshift $\C \subseteq C^{\FIN_A}$ is \emph{minimal} if there is no subshift $\D \subsetneq \C$.

\begin{lemma}[$\WKL_0$]\label[lemma]{lem:words-minimal-subshift-ur}
If~$\C \subseteq C^{\FIN_A}$ is a minimal subshift and~$f \in \C$, then $f$ is uniformly recurrent.
\end{lemma}
\begin{proof}
Suppose~$f$ is not uniformly recurrent. Then there is some~$\ell \in \NN$ such that for every bound~$m  > \ell$, there is some located word~$p \in \FIN_A(m, \infty)$ such that for every~$q \in \FIN_A(\ell, m)$, $S^\ell_{q \cup p}(f) \neq S^\ell(f)$.
Let~$\D$ be the $\Pi^0_1$ class of all~$g \in \C$ which does not contain~$S^\ell(f)$ as a factor. Note that $\D$ is a subshift.
As~$f \not \in \D$, we have that $\D \subsetneq \C$. By minimality assumption, then $\D = \emptyset$.
By $\WKL$, there is a bound~$m > \ell$ such that for every~$g \in \C$, there is some~$q \in \FIN_A(\ell, m)$ such that $S^\ell_q(g) = S^\ell(f)$.

Fix some $p \in \FIN_A(m, \infty)$. By definition of a subshift, there is some~$g \in \C$ such that $S^m(g) = S^m_p(f)$. In particuler, for every~$q \in \FIN_A(\ell, m)$, $S^\ell_q(g) = S^\ell_{q \cup p}(f)$.
Since~$g \in \C$, there is some~$q \in \FIN_A(\ell, m)$ such that $S^\ell_q(g) = S^\ell(f)$. But $S^\ell_q(g) = S^\ell_{q \cup p}(f)$, so $S^\ell_{q \cup p}(f) = S^\ell(f)$. This contradicts our initial choice of~$\ell$.
\end{proof}

From a purely mathematical viewpoint, the existence of minimal subshifts follow from Zorn's lemma. Day~\cite{day2016strength} however proved that their existence follow from~$\ACA_0$ in the context of colorings of~$\NN$. The proof of the existence of minimal subshifts in our case follows exactly the same argument.

\begin{lemma}[$\ACA_0$, Day~\cite{day2016strength}]\label[lemma]{lem:words-subshift-minimal-subshift}
Every subshift contains a minimal subshift.
\end{lemma}
\begin{proof}
Let~$\C \subseteq C^{\FIN_A}$ be a subshift. Fix an enumeration of all finite colorings $h_0, h_1, \dots$ of the form $\FIN_A(0, \ell) \to C$ for some~$\ell \in \NN$. Define a sequence $\C_0 \supseteq \C_1 \supseteq \dots$ of subshifts as follows: $\C_0 = \C$. Given~$\C_i$, let~$\D$ be the class of~$f \in \C_i$ which does not contain~$h_i$ as a factor. If~$\D \neq \emptyset$, then $\C_{i+1} = \D$, otherwise $\C_{i+1} = \C_i$. Then $\bigcap_i \C_i$ is a minimal subshift.
\end{proof}

\subsection{Orbit closures}

The notion of orbit closure of a coloring is central in topological dynamics. There is no clear definition of orbit in the case of colorings of located words, but the notion of orbit closure can be directly defined.

\begin{definition}
The \emph{orbit closure} of $f \in C^{\FIN_A}$ is the collection $\overline{\Orb(f)}$ of all~$g \in C^{\FIN_A}$
such that for every~$\ell \in \NN$, and every~$q \in \FIN_A(\ell, \infty)$, there is some~$p \in \FIN_A(\ell, \infty)$ such that $S^\ell_q(g) = S^\ell_p(f)$.
\end{definition}

In particular, $\overline{\Orb(f)}$ is a $\Pi^0_1(f')$ class and is the class of path of the following tree:
$$
T_f = \bigcup_m \left\{ h \in C^{\FIN_A(0, m)} : \begin{array}{l}
 		\forall \ell < m \forall q \in \FIN_A(\ell, m)\\
 		\exists p \in \FIN_A(\ell, \infty)\ S^\ell_p(f) = S^\ell_q(h)
 \end{array}
 \right\}
$$
Given a finite coloring $h : \FIN_A(0, \ell) \to C$, we write $[h]$ for the class of all~$f \in C^{\FIN_A}$ such that $S^\ell(f) = h$.

\begin{lemma}[$\RCA_0$]
For every~$f$, $\overline{\Orb(f)}$ is a subshift.
\end{lemma}
\begin{proof}
$\overline{\Orb(f)}$  is clearly closed as $\overline{\Orb(f)} = [T_f]$.

Fix some~$g \in \overline{\Orb(f)}$, $\ell \in \NN$ and $p \in \FIN_A(\ell, \infty)$. 
Let us show that there is some~$h \in \overline{\Orb(f)}$ such that $S^\ell(h) = S^\ell_p(g)$.
Let~$h$ be defined by $h(q) = g(q \cup p)$ for~$q \in \FIN_A(0, \ell)$, and $h(q) = g(q)$ otherwise. By definition, $S^\ell(h) = S^\ell_p(g)$.
We claim that $h \in \overline{\Orb(f)}$.
Fix some~$\ell_1 \in \NN$, some~$q \in \FIN_A(\ell_1, \infty)$. If~$q \in \FIN_A(0, \ell)$, then $S^{\ell_1}_q(h) = S^{\ell_1}_{q \cup p}(g)$. Since~$g \in \overline{\Orb(f)}$, there is some~$p_1 \in \FIN_A(\ell_1, \infty)$ such that $S^{\ell_1}_{q \cup p}(g) = S^{\ell_1}_{p_1}(f)$. In particular $S^{\ell_1}_q(h) = S^{\ell_1}_{p_1}(f)$.
If $q \not \in \FIN_A(0, \ell)$, then $S^{\ell_1}_q(h) = S^{\ell_1}_q(g)$. Again, since $g \in \overline{\Orb(f)}$, there is some~$p_1 \in \FIN_A(\ell_1, \infty)$ such that $S^{\ell_1}_q(g) = S^{\ell_1}_{p_1}(f)$.
In both cases, $S^{\ell_1}_q(h) = S^{\ell_1}_{p_1}(f)$ for some~$p_1 \in \FIN_A(\ell_1, \infty)$.
\end{proof}

\begin{lemma}[$\ACA_0$, Kreuzer~\cite{kreuzer2013minimal}]\label[lemma]{lem:ur-orbit-minimal}
If~$f$ is uniformly recurrent, then $\overline{\Orb(f)}$ is minimal.
\end{lemma}
\begin{proof}
Suppose that~$\overline{\Orb(f)}$ is not minimal. Then there is some~$g \in \overline{\Orb(f)}$
such that $\overline{\Orb(g)} \subsetneq \overline{\Orb(f)}$. In particular, $f \not \in \overline{\Orb(g)}$, so there some~$\ell \in \NN$ such that $\overline{\Orb(g)} \cap [S^\ell(f)] = \emptyset$.

Since~$f$ is uniformly recurrent, there is a bound~$m > \ell$ such that for every~$p \in \FIN_A(m, \infty)$, there is some~$q \in \FIN_A(\ell, m)$ such that $S^\ell_{q \cup p}(f)  = S^\ell(f)$. Since~$g \in \overline{\Orb(f)}$, there is some~$p \in \FIN_A(m, \infty)$ such that $S^m(g) = S^m_p(f)$. In particular, for every $q \in \FIN_A(\ell, m)$, $S^\ell_q(g) = S^\ell_{q \cup p}(f)$. Let~$q \in \FIN_A(\ell, m)$ be such that $S^\ell_{q \cup p}(f) = S^\ell(f)$. Then $S^\ell_q(g) = S^\ell_{q \cup p}(f) = S^\ell(f)$. But by definition of a subshift, there is some~$h \in \overline{\Orb(Y)}$ such that $S^\ell(h) = S^\ell_q(g)$, contradicting the fact that $\overline{\Orb(g)} \cap [S^\ell(f)] = \emptyset$.
\end{proof}

\begin{lemma}[$\ACA_0$, Day~\cite{day2016strength}]\label[lemma]{lem:orb-minimal-ur}
If~$\overline{\Orb(f)}$ is minimal, then $f$ is uniformly recurrent.
\end{lemma}
\begin{proof}
$\overline{\Orb(f)}$ is a minimal subshift containing~$f$. By \Cref{lem:words-minimal-subshift-ur},
$f$ is uniformly recurrent.
\end{proof}

\begin{lemma}[$\ACA_0$]\label[lemma]{lem:word-existence-minimal-in-orbit-closure}
For every~$f \in C^{\FIN_A}$, there is a coloring~$g \in \overline{\Orb(f)}$ such that $\overline{\Orb(g)}$ is minimal.
\end{lemma}
\begin{proof}
By \Cref{lem:words-subshift-minimal-subshift}, $\overline{\Orb(f)}$ contains a minimal subshift~$\D \subseteq \overline{\Orb(f)}$. Let~$g \in \D$. In particular, $\overline{\Orb(g)} \subseteq \D$ is a subshift. By minimality of~$\D$, $\overline{\Orb(g)} = \D$.
\end{proof}

\bigskip
\section{Auslander-Ellis theorem for $C^{\FIN_A}$}\label[section]{sect:aet}

We have see that recurrent colorings are simple instances of Carlson's theorem for located words. Although colorings are not recurrent in general, the Auslander-Ellis theorem says that every coloring is close to a recurrent coloring. Here again, due to the non-linearity of the ordering of words, there exists multiple notions of proximality.

\begin{definition}
Fix two colorings~$f, g : \FIN_A \to C$.
\begin{enumerate}
	\item $f$ and $g$ are \emph{weakly proximal} if for every~$\ell\in \NN$,
there is a located word~$p \in \FIN_A(\ell, \infty)$ such that $S^\ell_p(g) = S^\ell_p(f)$.
	\item $f$ and $g$ are \emph{proximal} if for every~$\ell\in \NN$,
there is a located variable word~$p \in \FIN_{A \varp}(\ell, \infty)$ such that
for each~$a \in A$, $S^\ell_{p[a]}(g) = S^\ell_{p[a]}(f)$.
	\item $g$ is \emph{strongly proximal} to~$f$ if for every~$\ell \in \NN$,
there is a located variable word~$p \in \FIN_{A \varp}(\ell, \infty)$ such that for each~$a \in A$,
$S^\ell(g) = S^\ell_{p[a]}(g) = S^\ell_{p[a]}(f)$.
\end{enumerate}
\end{definition}

Note that if $f$ is recurrent, then it is strongly proximal to itself.
There exists multiple candidate statements for adapting the Auslander-Ellis theorem to~$C^{\FIN_A}$ since there exists three notions of proximality. Two of them are of interest:

\begin{theorem}[Auslander-Ellis for $C^{\FIN_A}$]\label[theorem]{thm:aet-for-words}
For every coloring~$f : \FIN_A \to C$, there is a uniformly recurrent coloring~$g : \FIN_A \to C$ weakly proximal to~$f$.
\end{theorem}

\begin{theorem}[Weak Auslander-Ellis for $C^{\FIN_A}$]\label[theorem]{thm:weak-aet-for-words}
For every coloring~$f : \FIN_A \to C$, there is a coloring~$g : \FIN_A \to C$ strongly proximal to~$f$.
\end{theorem}

In what follows, let $HJ(k, \ell)$ be an integer large enough so that for every~$h : \FIN_A(HJ(k, \ell)) \to \ell$, there is a located variable word~$p \in \FIN_{A \varp}(HJ(k, \ell))$ such that $[p]_A$ is $h$-homogeneous.
The following lemma shows that the Auslander-Ellis theorem for $C^{\FIN_A}$ implies its weak version.


\begin{lemma}[$\RCA_0$]\label[lemma]{lem:vw-ur-wp-strongly-proximal}
If~$g$ is uniformly recurrent and weakly proximal to~$f$, then it is strongly proximal to~$f$.
\end{lemma}
\begin{proof}
Fix~$\ell \in \NN$. Since~$g$ is uniformly recurrent, there is some~$m > \ell$ be such that for every~$p \in \FIN_A(m, \infty)$, there is some~$q \in \FIN_A(\ell, m)$ such that $S^\ell_{q \cup p}(g) = S^\ell(g)$.
Let~$N = HJ(k, |\FIN_A(\ell, m)|)$.
Since~$g$ is weakly proximal to~$f$, there is some~$p \in \FIN_A(m+N, \infty)$ such that $S^{m+N}_p(g) = S^{m+N}_p(g)$.
Let~$h : \FIN_A(m, m+N) \to \FIN_A(\ell, m)$ be defined by $h(v) = q$ such that $S^\ell_{q \cup v \cup p}(g) = S^\ell(g)$.  By the finite Hales-Jewett theorem, there is a located variable word~$u \in \FIN_A(m, m+N)$ and some~$q \in \FIN_A(\ell, m)$ such that for every~$a \in A$, $h(u[a]) = q$.
In other words, for every~$a \in A$, $S^\ell_{q \cup u[a] \cup p}(g) = S^\ell(g)$.
Note that $S^\ell_{q \cup u[a] \cup p}(g) = S^\ell_{q \cup u[a] \cup p}(f)$.
Thus, letting~$w = q \cup u[i] \cup p$, for every~$a \in A$, $S^\ell(g) = S^\ell_{w[a]}(g) = S^\ell_{w[a]}(f)$.
\end{proof}

\begin{corollary}[$\RCA_0$]
The Auslander-Ellis theorem for~$C^{\FIN_A}$ implies its weak version.
\end{corollary}
\begin{proof}
Immediate by \Cref{lem:vw-ur-wp-strongly-proximal}.
\end{proof}

\subsection{Proof of Carlson's theorem for located words}

The same way Furstenberg and Weiss~\cite{furstenberg1978topological} proved that the original Auslander-Ellis theorem implies Hindman's theorem, we will prove that the weak version of the Auslander-Ellis theorem for  $C^{\FIN_A}$ implies Carlson's theorem for located words.


\begin{lemma}[$\RCA_0$]\label[lemma]{lem:vw-strong-prox-hom}
Suppose $f, g : \FIN_A \to C$ are colorings such that $g$ is strongly proximal to~$f$.
Then there is an infinite block sequence $X \subseteq \FIN_{A \varp}$ such that $[X]_A$ is both $f$-homogeneous and $g$-homogeneous for color~$g(\emptyset)$.	
\end{lemma}
\begin{proof}
We build a sequence~$p_0 < p_1 < \dots$ of located variable words inductively.

Since~$g$ is strongly proximal to~$f$, there is some located variable word~$p_0 \in \FIN_{A \varp}$ such that for each~$a \in A$, $S^0_{p_0[a]}(f) = S^0_{p_0[a]}(g) = S^0(g)$.
In particular, $f(p_0[a]) = g(p_0[a]) = g(\emptyset)$ for every~$a \in A$.

Assume~$p_0 < \dots < p_n$ are located variable words such that $[p_0, \dots, p_n]_A$ is both $f$-homogeneous and $g$-homogeneous for color~$g(\emptyset)$. Let~$\ell = 1 \max \dom p_n$.
Since~$g$ is strongly proximal to~$f$, there is some located variable word~$p_{n+1} \in \FIN_{A \varp}(\ell, \emptyset)$ such that for each~$a \in A$, $$S^\ell_{p_{n+1}[a]}(f) = S^\ell_{p_{n+1}[a]}(g) = S^\ell(g)$$
Let~$p \in [p_0, \dots, p_n]_A \cup \{\emptyset\}$. Note that $\max \dom p < \ell$ and that $g(p) = g(\emptyset)$. Then for every~$a \in A$,
$$f(p \cup p_{n+1}[a]) = S_{p_{n+1}[a]}(f)(p) = g(p) = g(\emptyset)$$
$$g(p \cup p_{n+1}[a]) = S_{p_{n+1}[a]}(g)(p) = g(p) = g(\emptyset)$$
Thus, $[p_0, \dots, p_n, p_{n+1}]_A$ is both $f$-homogeneous and $g$-homogeneous for color~$g(\emptyset)$.
\end{proof}

\begin{corollary}[$\RCA_0$]
The weak Auslander-Ellis theorem for $C^{\FIN_A}$ implies Carlson's theorem for located words.
\end{corollary}
\begin{proof}
Immediate by \Cref{lem:vw-strong-prox-hom}.
\end{proof}

\subsection{Proof of the Auslander-Ellis theorem for~$C^{\FIN_A}$}

The purpose of this section is to prove the Auslander-Ellis theorem for~$C^{\FIN_A}$ from an iterated version of the Finite Union theorem. Together with the previous section, we will obtain a proof of Carlson's theorem for located words from the iterated Finite Union theorem. The proof follows the structure of Section~5 of Blass, Hirst and Simpson~\cite{blass1987logical}.

We can actually refine our statement about the Finite Union theorem and consider only bounded unions. Given some~$r \in \NN$, let
$$
\FU^{\leq r}(X) = \{ \cup F : F \subseteq X \wedge 0 < |F| \leq n \}
$$

\begin{theorem}[Iterated Finite Union (bounded version)]
Let $X \subseteq \P_f(\NN)$ be an infinite block sequence.
For every sequence $(f_n)_{n \in \NN}$ of colorings $\FU(X) \to C$, there is an infinite block sequence $Y \subseteq \FU(X)$ such that for every~$n$, there is some finite set~$F \subseteq Y$ such that $\FU^{\leq r}(Y - F)$ is $f_n$-homogeneous.
\end{theorem}

We will write~$\IFUT^{\leq r}$ as a shorthand for the iterated Finite Union theorem for unions of at most~$r$ elements.
We need a notion of weak block sequence playing the same role as the notion of block sequence, but for located words instead of located variable words.

\begin{definition}
A \emph{weak block sequence} is a totally ordered set~$X \subseteq \FIN_A$.
Given a weak block sequence~$X$ and some~$r \in \NN$, we let
$$[X]^{\leq r}_A = \{ p_0 \cup \dots \cup p_{k-1} \in \FIN_A : k \leq r \wedge  p_0, \dots, p_{k-1} \in \FIN_A \}$$
\end{definition}

The following notion is an adaptation of the IP-limit to combinatorial spaces of located words.

\begin{definition}
Given~$f, g : \FIN_A \to C$ and an infinite weak block sequence~$X \subseteq \FIN_A$, we write $\fulim_{X}(f) = g$ if for every~$\ell \in \NN$, there is a finite set~$F \subseteq X$ such that for every~$p \in [X - F]^{\leq 2}_A$, $S^\ell_p(f) = S^\ell(g)$.	
\end{definition}

\begin{lemma}[$\RCA_0$]\label[lemma]{lem:vw-flim-strong-prox}
If~$g = \fulim_X(f)$ for an infinite weak block sequence $X \subseteq \FIN_A$ then~$g$ is weakly proximal to~$f$.
\end{lemma}
\begin{proof}
Suppose $g = \fulim_X(f)$. Given~$\ell \in \NN$, let~$F_0 \subseteq X$ be such that for every~$p \in [X - F_0]^{\leq 2}_A$, $S^\ell_p(f) = S^\ell(g)$. Fix~$p \in X - F_0$ and let~$m = 1 + \max p$. Let $F_1 \subseteq X$ be such that for every~$q \in [X - F_1]^{\leq 2}_A$, $S^m_q(f) = S^m(g)$. Fix~$q \in X - F_1$.
In particular, $S^\ell_{p \cup q}(f) = S^\ell_p(g)$. Since~$p, p \cup q \in [X - F_0]^{\leq 2}_A$, then $S^\ell_p(f) = S^\ell(g)$ and $S^\ell_{p \cup q}(f) = S^\ell(g)$. It follows that
$$
S^\ell_p(f) = S^\ell(g) = S^\ell_{p \cup q}(f) = S^\ell_p(g)
$$
In other words, $g$ is weakly proximal to~$f$.
\end{proof}

\begin{lemma}[$\RCA_0 + \IFUT^{\leq 2}$]\label[lemma]{lem:vw-fs2-limit-exists}
For every coloring~$f : \FIN_A \to C$ and every infinite weak block sequence~$X \subseteq \FIN_A$, there is an infinite weak block sequence~$Y \subseteq [X]_A$ such that $\fulim_{Y}(f)$ exists.
\end{lemma}
\begin{proof}
Let~$B = \{ \dom p : p \in X \}$. There exists a canonical bijection $\pi : \FU(B) \to [X]_A$.

For every $\ell \in \NN$ and finite coloring~$h : \FIN_A(0, \ell) \to C$, let $C_h = \{ F \in \FU(B) : S^\ell_{\pi(F)}(f) = h \}$.
Thus $\langle C_h : h \in \bigcup_\ell C^{\FIN_A(0, \ell)} \rangle$ is a countable sequence of 2-colorings of~$\FU(B)$.
By $\IFUT^{\leq 2}$, there is an infinite set~$D \subseteq \FU(B)$ such that
for every $h : \FIN_A(0, \ell) \to C$, there is a finite set~$G_h \subseteq D$ such that $\FU^{\leq 2}(D - G_h)$ is homogeneous for~$C_h$. Moreover, the sequence $\langle G_h : h \in \bigcup_\ell C^{\FIN_A(0, \ell)} \rangle$ exists by~$\ACA_0$, which follows from $\RCA_0 + \IFUT^{\leq 2}$.

We claim that for every~$\ell \in \NN$, there is at most one $h : \FIN_A(0, \ell) \to C$ such that
$\FU^{\leq 2}(D - G_h) \subseteq C_h$. Indeed, if there is are $h_0, h_1 : \FIN_A(0, \ell) \to C$ such that $\FU^{\leq 2}(D - G_{h_i}) \subseteq C_{h_i}$ for each~$i < 2$, then pick some~$F \in D - (G_F \cup G_H)$. Then $F \in C_{h_0} \cap C_{h_1}F$, so by definition of~$C_{h_i}$, $S^\ell_{\pi(F)}(f) = h_i$, so $h_0 = h_1$.

For every~$\ell \in \NN$, let~$h_\ell : \FIN_A(0, \ell) \to C$ witness the claim. Note that $h_\ell \subseteq h_{\ell+1}$. The sequence $\langle h_\ell : \ell \in \NN  \rangle$ can obtained computably from the sequence $\langle G_h : h \in \bigcup_\ell C^{\FIN_A(0, \ell)} \rangle$. Let~$g : \FIN_A \to C$ be such that $S^\ell(g) = h_\ell$ for every~$\ell \in \NN$. Let~$Y = \{\pi(F) : F \in D \}$. Note that $Y \subseteq [X]_A$ is a weak block sequence. By construction, $\fulim_{Y}(f) = g$.
\end{proof}

\begin{lemma}[$\RCA_0 + \IFUT^{\leq 2}$]\label[lemma]{lem:vw-fs2-ur-limit-exists}
	For every coloring~$f : \FIN_A \to C$ and every~$g \in \overline{\Orb(f)}$, there is a coloring~$h \in \overline{\Orb(g)}$ such that $h = \fulim_X(f)$ for some infinite weak block sequence~$X \subseteq \FIN_A$.
\end{lemma}
\begin{proof}
Fix~$f$ and $g$. Assume $f \neq g$, otherwise the result follows from \Cref{lem:vw-fs2-limit-exists} directly.

For every~$\ell \in \NN$, let~$\V_\ell = C^{\FIN_A(0, \ell)}$, and let~$\V = \bigcup_\ell \V_\ell$.
Let~$T \subseteq \V$ be a code for the subshift $\overline{\Orb(g)}$ such that~$T$ has no leaves. For each~$\ell \in \NN$, let~$T_\ell = T \cap \V_\ell$. Let~$\U_\ell$ be the open class induced by the leaves of~$T_\ell$. Last, for every~$p \in \FIN_A(\ell, \infty)$,
let~$\W_{\ell, p} = \{ h \in C^{\FIN_A} : S^\ell_p(h) \in T_\ell \}$.
Note that for all $\ell \in \NN$ and $p \in \FIN_A(\ell, \infty)$, $\overline{\Orb(g)} \subseteq \W_{\ell, p}$ and $\W_{\ell, p}$ is open.

Define the weak block sequence~$p_0 < p_1 < \dots$ inductively as follows. Let~$p_0 \in \FIN_A$ be arbitrary.
Given~$p_\ell$ defined, with~$m_\ell = 1 + \max \dom p_\ell$, since~$\bigcap_{p \in \FIN_A(0, m_\ell)} \W_{\ell,p}$ is an open cover of $\overline{\Orb(g)}$, there is some bound~$b \in \NN$ such that $\U_b \subseteq \bigcap_{p \in \FIN_A(0, m_\ell)} \W_{\ell,p}$.
By definition of~$T_b$, since~$g \in \Orb(g)$, there is some~$h \in T_b$ be such that $S^b(g) = h$. In particular, $[h] \subseteq \U_b$.
 Since~$g \in \overline{\Orb(f)}$, then there is some located word~$p_{\ell+1} \in \FIN_A(b, \infty)$ such that $S^b_{p_{\ell+1}}(f) = h$. Moreover, since $f \neq g$, the located word~$p_{\ell+1}$ can be taken to be non-empty.
In particular, $[S^b_{p_{\ell+1}}(f)] \subseteq \U_b \subseteq \bigcap_{p \in \FIN_A(0, m_\ell)} \W_{\ell,p}$.
Thus for each~$\ell \in \NN$, $S^\ell_p(f) \in T_\ell$ for all sufficiently large~$p \in [p_0, p_1, \dots]_A$.

By \Cref{lem:vw-fs2-limit-exists}, there is an infinite weak block sequence~$X \subseteq [p_0, p_1, \dots]_A$ such that $h =  \fulim_X(f)$ exists. Then $h \in \bigcap_\ell \U_\ell = \overline{\Orb(Y)}$.
\end{proof}

We are now ready to prove the Auslander-Ellis theorem for~$C^{\FIN_A}$ from the iterated Finite Union theorem.

\begin{proof}[Proof of~\Cref{thm:aet-for-words} over $\RCA_0 + \IFUT^{\leq 2}$]
Let~$f : \FIN_A \to C$ be a coloring. By \Cref{lem:words-subshift-minimal-subshift}, let~$g \in \overline{\Orb(f)}$ be such that $\overline{\Orb(g)}$ is minimal. By \Cref{lem:vw-fs2-ur-limit-exists},
there is a coloring~$h \in \overline{\Orb(g)}$ and an infinite weak sequence~$X \subseteq \FIN_A$ such that $h = \fulim_X(f)$. By \Cref{lem:vw-flim-strong-prox}, $h$ is weakly proximal to~$f$. Moreover, $\overline{\Orb(h)} = \overline{\Orb(h)}$ by minimality, hence~$h$ is uniformly recurrent by \Cref{lem:words-minimal-subshift-ur}.
\end{proof}

This gives us in particular a proof of Carlson's theorem from the iterated Finite Union theorem for unions of at most 2 elements.

\begin{proof}[Proof of~\Cref{thm:carlson-theorem} over $\RCA_0 + \IFUT^{\leq 2}$]
Let~$f : \FIN_A \to C$ be a coloring. By \Cref{thm:aet-for-words}, there is a coloring $g$ which is uniformly recurrent, and weakly proximal to~$f$. By \Cref{lem:vw-ur-wp-strongly-proximal}, $g$ is strongly proximal to~$f$. By \Cref{lem:vw-strong-prox-hom}, there is an infinite block sequence~$X \subseteq \FIN_{A \varp}$ such that $[X]_A$ is $f$-homogeneous.
\end{proof}

\begin{corollary}\label[corollary]{cor:bounded-fut-implies-fut}
$\RCA_0 \vdash \IFUT^{\leq 2} \rightarrow \FUT$
\end{corollary}
\begin{proof}
Immediate, since $\RCA_0 + \IFUT^{\leq 2}$ proves Carlson's theorem for located words (\Cref{thm:carlson-theorem}), which itself implies the Finite Union theorem.
\end{proof}

It is a major open question is combinatorics whether Hindman's theorem for bounded sums implies Hindman's theorem (see Hindman, Leader and Strauss~\cite[Question 12]{hindman2003open}).
\Cref{cor:bounded-fut-implies-fut} can be seen as a partial negative answer.

\section{Open questions}\label[section]{sect:open-questions}

There are many remaining open questions in the reverse mathematics of Ramsey's theory. We mention a few of them. The first one, maybe the most important, is the following.

\begin{question}
Does (iterated) Hindman's theorem hold in~$\ACA_0$?
\end{question}

We have seen that the iterated Finite Union theorem implies Carlson's theorem for located words.
However, the proof does not seem to be adaptable to replace $\IFUT$ with its non-iterated version.

\begin{question}
Does Hindman's theorem imply Carlson's theorem for (located) words?
\end{question}

Last, as mentioned in the introduction, Carlson proved a stronger statement about variable words, which has no known elementary proof.

\begin{question}
What are the reverse mathematics of Carlson's theorem for variable (located) words?	
\end{question}

\section*{Acknowledgement}

Patey is partially supported by grant ANR ``ACTC'' \#ANR-19-CE48-0012-01.

\bibliographystyle{plain}
\bibliography{bibliography}

\end{document}